\documentclass[a4paper,11pt,reqno]{amsart}
\usepackage[a4paper,centering,scale=0.7]{geometry}
\usepackage{amssymb}

\newtheorem{prop}{Proposition}[section]
\newtheorem{thm}[prop]{Theorem}
\newtheorem{lemma}[prop]{Lemma}
\newtheorem{coroll}[prop]{Corollary}

\theoremstyle{remark}
\newtheorem{rmk}[prop]{Remark}

\newcommand{\ep}[1]{{#1}^{(\varepsilon)}}
\newcommand{\ip}[2]{\langle #1,#2 \rangle}
\renewcommand{\div}{\operatorname{div}}
\newcommand{\lip}{\dot{C}^{0,1}}
\newcommand{\wto}{\rightharpoonup}

\newcommand{\E}{\mathbb{E}}
\newcommand{\enne}{\mathbb{N}}
\renewcommand{\P}{\mathbb{P}}
\newcommand{\erre}{\mathbb{R}}

\begin{document}

\title[Ergodicity for stochastic equations with jumps]{Ergodicity for
  nonlinear stochastic evolution equations with multiplicative Poisson
  noise}

\author{Carlo Marinelli}
\address[C.~Marinelli]{Institut f\"ur Angewandte Mathematik, Universit\"at Bonn,
  Endenicher Allee 60, D-53115 Bonn, Germany}
\urladdr{http://www.uni-bonn.de/$\sim$cm788}

\author{Giacomo Ziglio}
\address[G.~Ziglio]{Dipartimento di Matematica, Universit\`a di
Trento, via Sommarive 14, I-38123 Trento, Italy}

\date{20 September 2009}

\begin{abstract}
  We study the asymptotic behavior of solutions to stochastic
  evolution equations with monotone drift and multiplicative Poisson
  noise in the variational setting, thus covering a large class of
  (fully) nonlinear partial differential equations perturbed by jump
  noise. In particular, we provide sufficient conditions for the
  existence, ergodicity, and uniqueness of invariant measures.
  Furthermore, under mild additional assumptions, we prove that the
  Kolmogorov equation associated to the stochastic equation with
  additive noise is solvable in $L_1$ spaces with respect to an
  invariant measure.
\end{abstract}



\thanks{The work for this paper was carried out while the first-named
  author was visiting the Department of Statistics of Purdue
  University supported by a MOIF fellowship.}

\maketitle

\section{Introduction}
This paper is devoted to the study of asymptotic properties of the
solution to an infinite dimensional stochastic differential equation
of the type
\begin{equation}     \label{eq:caro}
\left\{
\begin{aligned}
&du(t)+Au(t)dt=\int_Z G(u(t-),z)\,\bar{\mu}(dt,dz)\\
&u(0)=x
\end{aligned}
\right.
\end{equation}
where $A$ is a nonlinear monotone operator defined on an evolution
triple $V \subset H \subset V'$ (see e.g. the classical works
\cite{Pard, KR-spde}), and $\bar{\mu}$ is a compensated
Poisson measure.  Precise assumptions on the data of the problem will
be given below.  In particular, $A$ may be chosen as the $p$-Laplace
operator, as well as the porous media diffusion operator
$-\Delta\beta(\cdot)$, thus covering a wide class of nonlinear partial
differential equations with discontinuous random perturbations.

While existence and uniqueness of solutions for (\ref{eq:caro}) has
been established in \cite{Gyo-semimg} (in fact allowing $\bar{\mu}$ to
be a general compensated random measure), we are not aware of any
result on the asymptotic behavior of the solutions to such equations.
Furthermore, as we show in this paper, invariant measures provide a
suitable class of reference measures with respect to which one can
study infinite dimensional Kolmogorov equations of non-local type,
thus extending results that, to the best of our knowledge, were
available only for second-order (local) Kolmogorov equations (see e.g.
\cite{DP-K}).

Let us briefly describe our main results in more detail: we first
prove the existence of an invariant measure for the Markovian
semigroup associated to (\ref{eq:caro}), under the (standing)
assumption that $V$ is compactly embedded in $H$. Moreover, suitable a
priori estimates on any invariant measure imply the existence of an
ergodic invariant measure, and an extra superlinearity assumption on
$A$ yields exponential mixing, hence uniqueness. Finally, we prove
that the (non-local) Kolmogorov operator $L$ associated to
(\ref{eq:caro}), with $G$ independent of $u$, is essentially
$m$-dissipative in $L_1(H,\nu)$, with $\nu$ an infinitesimally
invariant measure for $L$. The last result in particular is equivalent
to the solvability in $L_1(H,\nu)$ of the (elliptic)
integro-differential Kolmogorov equation associated to
(\ref{eq:caro}).

We should mention that the case where the right-hand side in
(\ref{eq:caro}) is replaced by an additive Gaussian noise has been
considered in \cite{BDP-erg}, where sufficient conditions for the
existence and the uniqueness of invariant measures are given.
Moreover, the authors study the Kolmogorov equation associated to
(\ref{eq:caro}) in $L_2(H,\nu)$, assuming that $A$ is differentiable
and its differential satisfies a certain polynomial growth condition.
Our $L_1$ approach does not require any such hypothesis. Moreover,
combining the results in \cite{BDP-erg} with ours and appealing to the
L\'evy-It\^o decomposition theorem, one could rather easily obtain
corresponding results for evolution equations driven by general
(locally) square-integrable L\'evy noise.

In this regard, let us also recall that results on existence and
uniqueness of invariant measures for semilinear evolution equations
driven by L\'evy noise can be found in the recent monograph
\cite{PZ-libro}, as well as in \cite{cm:rd}. However, the authors work
in the mild setting, hence equations with fully nonlinear drift (i.e.
without a leading linear operator generating a strongly continuous
semigroup) cannot be covered. 

The rest of the paper is organized as follows: results on existence,
uniqueness, and ergodicity of invariant measures $\nu$ are contained
in Section \ref{sec:inv.meas}. In Section \ref{sec:Kol}, assuming that
$G$ does not depend on $u$ and that $A$ satisfies a (mild)
``regularizability'' hypothesis, we prove that the Kolmogorov operator
associated to the stochastic equation (\ref{eq:caro}) is dissipative,
hence closable, and its closure is $m$-dissipative in $L_1(H,\nu)$.
Equivalently, this amounts to saying that the (elliptic)
infinite-dimensional non-local Kolmogorov equation associated to
(\ref{eq:caro}) is uniquely solvable in $L_1(H,\nu)$.  In Section
\ref{sec:ex} we show that our abstract results apply to several
situations of interest. In particular, we concentrate on equations
with non-linear drift in divergence form (thus including the
$p$-Laplace operator) and on the generalized porous media equations
with pure-jump noise.

\subsection{Notation}
Given a Banach (or Hilbert) space $E$, its norm will be denoted by
$|\cdot|_E$. We shall denote the space of all Borel measureable
bounded functions from $X$ to $\erre$ by $B_b(E)$. Given another
Banach space $F$, the space of $k$-times continuously differentiable
functions from $E$ to $F$ will be denoted by $C^k(E \to F)$, and
$C^{k,1}(E \to F)$ stands for the subset of $C^k(E \to F)$ whose
elements posses a Lipschitz continuous $k$-th derivative. We shall add
a subscript $\cdot_b$ if the functions themselves and all their
derivatives are bounded. If $\phi:E \to F$ is Lipschitz continuous, we
shall write $\phi \in \dot{C}^{0,1}(E \to F)$, and we define
\[
|\phi|_{\dot{C}^{0,1}(E \to F)} := 
\sup_{x,y\in E, x\neq y} \frac{|\phi(x)-\phi(y)|_F}{|x-y|_E}.
\]
If $F=\erre$, we shall simply write $C^k(E)$ etc. Sometimes we shall
just write $C^k$ etc. if it is obvious what $E$ and $F$ are.
By $\mathcal{M}_1(E)$ we shall indicate the space of probability
measures on $E$, endowed with the $\sigma(\mathcal{M}_1(E),C^0_b(E))$
topology.
Weak convergence (of functions and measures) will be denoted by by
$\wto$, without explicit reference to the underlying topology if
no confusion arises.

If $X \leq NY$ for some positive constant $N$, we shall equivalently
write $X \lesssim Y$. If $N$ depends on a set of parameters
$p_1,\ldots,p_n$, we shall also write $N=N(p_1,\ldots,p_n)$ and $X
\lesssim_{p_1,\ldots,p_n} Y$.

\section{Invariant measures and ergodicity}
\label{sec:inv.meas}
Let $(\Omega,\mathcal{F},(\mathcal{F}_t)_{t\geq 0},\P)$ be a filtered
probability space satisfying the usual conditions, and $\E$ denote
expectation with respect to $\P$. All stochastic elements will be
defined on this stochastic basis, unless otherwise specified.
Let $(Z,\mathcal{Z},m)$ be a measure space with a $\sigma$-finite
measure $m$ and $\mu$ a Poisson random measure on $\erre_+ \times Z$ with
compensator $\mathrm{Leb}\otimes m$, and set
$\bar{\mu}:=\mu-\mathrm{Leb}\otimes m$ (Leb stands for Lebesgue
measure on $\erre$).
Let $H$ be a real separable Hilbert space, and $G: H \times Z \to H$ a
measurable function such that
\[
|G(x,\cdot)|_m^2 := \int_Z |G(x,z)|_H^2\,m(dz) < \infty
\qquad \forall x \in H.
\]
Let $V$ and $V'$ be a reflexive Banach space and its dual,
respectively, such that $V \hookrightarrow H \hookrightarrow V'$ with
dense and continous embeddings. Thanks to Asplund's renorming theorem
\cite{asplund}, we shall assume without loss of generality that both
$V$ and $V'$ are strictly convex. Furthermore, we shall assume that $V
\hookrightarrow H$ is compact.  Both the duality pairing between $V$ and
$V'$ and the inner product in $H$ will be denoted by
$\ip{\cdot}{\cdot}$. 

The operator $A:V \to V'$ is assumed to be demicontinuous
(i.e. strongly-weakly closed) and to satisfy the monotonicity
condition
\begin{equation}\label{eq:mon}
2\ip{Ax-Ay}{x-y} - |G(x,\cdot)-G(y,\cdot)|_m^2 \geq 0
\qquad \forall x,\,y \in V,
\end{equation}
as well as the following coercivity and growth conditions:
\begin{align}
  2\ip{Ax}{x} - |G(x,\cdot)|_m^2 + \alpha_0|x|^2_H &\geq \alpha_1|x|^p_V
   -C_0 &\forall x\in V,   \label{eq:milch}\\
  |Ax|_{V'} &\leq C_1|x|_V^{p-1}+C_2 &\forall x\in V,   \label{eq:kofi}
\end{align}
for some constants $\alpha_0\geq0$, $\alpha_1>0$, $C_0,C_1>0$, $C_2
\in \erre$ and $p>2$. Instead of (\ref{eq:milch}) one could assume
that there exists a constant $\alpha_1>0$ such that
\[
2\ip{Ax}{x}- |G(x,\cdot)|_m^2 \geq \alpha_1|x|^2_V \qquad \forall x \in V.
\]
Note that, by (\ref{eq:milch}) and (\ref{eq:kofi}), one has
\begin{equation}     \label{weissbier}
  |G(x,\cdot)|^2_m\leq 2C_1|x|^p_V+\alpha_0|x|^2_H+2C_2|x|_V+C_0
\qquad \forall x \in V.
\end{equation}
All assumptions stated so far will be in force throughout the
paper and will be used without further mention.

Let us recall the following well-posedness result for (\ref{eq:caro})
due to Gy{\H o}ngy \cite[Thm.~2.10]{Gyo-semimg}. Here and in the
following we shall denote the space of $H$-valued random variables
with finite $p$-th moment by $\mathbb{L}_p(H)$, and the space of
adapted processes $X:[0,T] \to H$ such that $\E \sup_{t \leq T}
|X(t)|_H^p < \infty$ by $\mathbb{H}_p(T)$.
\begin{prop}
  Let $x \in \mathbb{L}_2(H)$ and $T \geq 0$. Then equation
  (\ref{eq:caro}) admits a unique strong solution $u$ such that $u(t)
  \in V$ $\mathbb{P}$-a.s. for a.a. $t \in [0,T]$, $t \mapsto u(t)$ is
  c\`adl\`ag in $H$, and satisfies
  \[
  \E\sup_{t \leq T} |u(t)|_H^2 + \E\int_0^T|u(t)|^p_V\,dt < \infty.
  \]
  Moreover, $u$ is a Markov process, and 
  the solution map $x \mapsto u$ is Lipschitz continuous
  from $\mathbb{L}_2$ to $\mathbb{H}_2(T)$.
\end{prop}
The solution to (\ref{eq:caro}) generates a Markovian semigroup $P_t$
on $B_b(H)$ by the usual prescription $P_t\phi(x)=\E\phi(u(t,x))$,
$\phi \in B_b(H)$. The continuity of the solution map ensures that
$P_t$ is Feller.

\smallskip

In the following subsection we establish the existence and uniqueness
of an ergodic invariant measure for $P_t$ under an assumption stronger
than (\ref{eq:mon}). The proof is adapted from a classical method used
for stochastic evolution equations with Wiener noise in the mild
setting (see e.g. \cite{DZ96}). This simple result is included only
for completeness, while the main results of this section are contained
in {\S}\ref{subs:gen.case}.

\subsection{Strictly dissipative case}\label{subs:str.diss.}
Throughout this subsection we assume that there exists
$\alpha\in(0,\infty)$ such that
\begin{equation}     \label{aperol}
2\ip{Ax-Ay}{x-y}-|G(x,\cdot)-G(y,\cdot)|_m^2
\geq \alpha |x-y|_H^2 \qquad \forall x,y\in V.
\end{equation}
We shall need a few preparatory results. The following inequality can
be obtained by a simple computation based on
(\ref{weissbier}), (\ref{aperol}), and Young's inequality (see e.g.
\cite[{\S}4.3]{PreRoeck} for a related case).
\begin{lemma}
  Let $\eta\in(0,\alpha)$. There exist $\delta_\eta\in(0,\infty)$
  such that
  \begin{equation} \label{sake}
    2\ip{Ax}{x}-|G(x,\cdot)|_m^2\geq\eta|x|_H^2-\delta_\eta
    \qquad \forall x \in V.
  \end{equation}
\end{lemma}

Let us define the random measure $\mu_1$ on $\erre \times Z$ as
\[
\mu_1(t,A) := 
\left\{
\begin{array}{ll}
\mu(t,A), & t \geq 0,\quad A \in \mathcal{Z},\\
\mu_0(-t,A), & t< 0,\quad A \in \mathcal{Z},
\end{array}
\right.
\]
with $\mu_0$ an independent copy of $\mu$, on the naturally associated
filtration $(\tilde{\mathcal{F}}_t)_{t\in\erre}$. Let us also define the
compensated measure measure $\tilde{\mu} := \mu_1 -
\mathrm{Leb}\otimes m$. 

For $s\in\erre$, consider the equation
\begin{equation}     \label{orange}
  \left\{
  \begin{aligned}
  &du(t)+Au(t)dt=\int_Z G(u(t-),z)\tilde{\mu}(dt,dz),\qquad t \geq s,\\
  &u(s)=x.
  \end{aligned}
  \right.
\end{equation}
It is clear that (\ref{orange}) admits a unique solution $u(t,s,x)$
which generates a semigroup $P_{s,t}$ on $B_b(H)$, exactly as above.

\begin{lemma}     \label{augustiner}
  Let $s\in(-\infty,0]$ and $x \in \mathbb{L}_2(H)$. There exists
  $\eta\in\mathbb{L}_2(H)$, independent of $x$, such that
  \[
  \lim_{s\rightarrow-\infty} \E |u(0,s,x) - \eta|_H^2 = 0.
  \]
  Moreover, one has
  \[
  \E|u(0,s,x)-\eta|_H^2 \lesssim e^{\alpha s}(1+\E|x|^2_H).
  \]
\end{lemma}
\begin{proof}
  For $s_1,s_2\in(-\infty,0]$, $s_1\leq s_2$ and $x\in
  \mathbb{L}_2(H)$, we have
  \begin{align*}
    u(0,s_1,x)-u(0,s_2,x)=&-\int_{s_2}^0[Au(r,s_1,x)-Au(r,s_2,x)]dr\\
    &+\int_{s_2}^0\int_Z [G(u(r,s_1,x),z)-G(u(r,s_2,x),z)]\tilde{\mu}(dr,dz)\\
    &+u(s_2,s_1,x)-x.
  \end{align*}
  Appealing to It\^o's formula for the square of the norm (see
  \cite{KryGyo2}) and recalling (\ref{sake}) we obtain
  \begin{equation}     \label{eq:trinka}
  \E|u(0,s_1,x)-u(0,s_2,x)|^2_H \leq
  \Big(\frac{\delta_\eta}{\eta}+2\E|x|_H^2\Big) e^{\alpha s_2}.
  \end{equation}
  Letting $s_2$ tend to $-\infty$, it follows that there exists
  $\eta(x)\in \mathbb{L}_2(H)$ such that
  \[
  \lim_{s\rightarrow-\infty}\E|u(0,s,x)-\eta(x)|^2=0.
  \]
  By the same arguments one can prove that
  \[
  \lim_{s\rightarrow-\infty}\E|u(0,s,x)-u(0,s,y)|^2=0
  \]
  for all $x,y\in \mathbb{L}_2(H)$, hence that $\eta$ is independent
  of $x\in \mathbb{L}_2(H)$. Letting $s_1$ tend to $-\infty$ in
  (\ref{eq:trinka}) one obtains the exponential convergence.
\end{proof}

We can now prove the main result of this subsection.
\begin{thm}     \label{paulaner}
  Assume that (\ref{aperol}) holds and that $x \in
  \mathbb{L}_2(H)$. There exists a unique invariant measure $\nu$ for
  $P_t$. Moreover, one has
  \[
  \int|y|^2_H\,\nu(dy) < \infty,
  \]
  and
  \[
  \left|P_t\varphi(y)-\int \varphi(w)\,\nu(dw)\right|\leq
  e^{-\frac\alpha2t} |\varphi|_{\dot{C}^{0,1}}\int|y-w|_H\,\nu(dw)
  \]
  for all $t \geq 0$, $y \in H$, and $\varphi \in C_b^{0,1}(H)$.
\end{thm}
\begin{proof}[Proof of Theorem \ref{paulaner}]
  Let $\nu$ be the law of the random variable $\eta$ constructed in
  Lemma \ref{augustiner}. In particular, $\int |y|^2\,\nu(dy)<\infty$
  is equivalent to $\eta \in \mathbb{L}_2(H)$. Similarly, the previous lemma
  immediately yields $P^*_{s,0}\delta_y \wto \nu$ for all $y \in H$ 
  in $\mathcal{M}_1(H)$ as $s \to -\infty$. Moreover, for any
  $\varphi \in C_b^{0,1}(H \to \erre)$, we have
  \[
  \int (P_{0,t}\varphi)\,d\nu =
  \lim_{s\rightarrow-\infty} \int(P_{0,t}\varphi)\,d(P^*_{s,0}\delta_y)
  = \lim_{s\rightarrow-\infty}(P_{s-t,0}\varphi)(y)
  = \int \varphi\,d\nu,
  \]
  i.e. $\nu$ is invariant for $P_t$. Moreover, if $\nu$ is an
  invariant measure for $P_t$, we have
  \[
  \Big| P_t\varphi(y) - \int\varphi\,d\nu \Big|
  = \Big| \int (P_t\varphi(y)-P_t\varphi(w))\,\nu(dw) \Big|
  \leq e^{-\frac{\alpha}{2}t} |\varphi|_{\lip} \int|y-w|_H\,\nu(dw)
  \]
  for all $t \geq 0$.
\end{proof}

\subsection{General case}\label{subs:gen.case}
We can still prove the existence of an ergodic invariant measure without
the assumption that the couple $(A,G)$ is strictly dissipative, using
an argument based on Krylov-Bogoliubov's theorem.
\begin{thm}     \label{thm:exim}
  There exists an invariant measure $\nu$ for $P_t$. Moreover,
  $\nu$ is concentrated on $V$, i.e. $\nu(V)=1$.
\end{thm}
\begin{proof}
  We assume $p>2$, since the proof for the case $p=2$ is completely
  similar. Let $x \in \mathbb{L}_2(H)$. By It\^o's formula for the
  square of the norm in $H$ (see \cite{KryGyo2}) we have
  \begin{align}
    |u(t,x)|^2_H &- |x|_H^2 = 2\int_0^t \ip{u(s-,x)}{du(s,x)} + [u](t)
    \nonumber\\
    &= -2\int_0^t \ip{Au(s,x)}{u(s,x)}\,ds
    +2\int_0^t\!\int_Z\ip{u(s-,x)}{G(u(s-,x),z)}\,\bar{\mu}(ds,dz)
    \nonumber\\
    &\quad +\int_0^t\!\int_Z|G(u(s-,x),z)|_H^2\,\mu(ds,dz).\label{eq:ito}
  \end{align}
  Taking expectations on both side and recalling that the compensator
  of $\mu$ is $\mathrm{Leb}\otimes m$, we obtain
  \begin{equation*}
    \E|u(t,x)|^2_H = -2\E\int_0^t \ip{Au(s,x)}{u(s,x)}\,ds + \E|x|_H^2
    + \int_0^t |G(u(s,x),\cdot)|^2_{m}\,ds,
  \end{equation*}
  hence, thanks to (\ref{eq:milch}),
  \begin{equation}     \label{lemonade}
    \E|u(t,x)|^2_H \leq \alpha_0\E\int_0^t|u(s,x)|_H^2\,ds
    -\alpha_1\E\int_0^t|u(s,x)|_V^p ds+\E|x|_H^2+tC_0.
  \end{equation}
  Since $V \hookrightarrow H$ is continuous, there exists a constant
  $c>0$ such that $|v|_H\leq c|v|_V$ for all $v \in V$, hence
  \[
  \E|u(t,x)|^2_H \leq \alpha_0\E\int_0^t|u(s,x)|_H^2\,ds -
  \frac{\alpha_1}{c^p} \E\int_0^t|u(s,x)|_H^p\,ds +\E|x|_H^2 + tC_0.
  \]
  The elementary inequality $\varepsilon^2 |y|^2 \leq \varepsilon^p
  |y|^p+1$ (with $\varepsilon>0$ and $p\geq2$) yields
  \[
  -\E|u(t,x)|_H^p \leq -\varepsilon^{2-p}\E|u(t,x)|_H^2 +
  \varepsilon^{-p},
  \]
  thus also
  \begin{align} \label{apfelschorle}
    \E|u(t,x)|^2_H &\leq -\left(\frac{\alpha_1\varepsilon^{2-p}}{c^p}
    - \alpha_0\right) \int_0^t\E|u(s,x)|_H^2\,ds 
    + t\left(\frac{\alpha_1\varepsilon^{-p}}{c^p}+C_0\right)+\E|x|_H^2\nonumber\\
    &=-\gamma\int_0^t\E|u(s,x)|_H^2 ds+\E|x|^2_H+tC
  \end{align}
  where
  \[
  \gamma := \frac{\alpha_1\varepsilon^{2-p}}{c^p}-\alpha_0, \qquad 
  C := \frac{\alpha_1\varepsilon^{-p}}{c^p}+C_0.
  \]
  Choosing $\varepsilon$ so that $\gamma>0$ and applying Gronwall's
  inequality to (\ref{apfelschorle}), it follows that
  \begin{equation} \label{apfelsaft}
    \E|u(t,x)|^2_H\leq\E|x|^2_He^{-\gamma t}+K\qquad\forall t\geq0
  \end{equation}
  where $K$ is a constant independent of $t$.  Moreover, by
  (\ref{lemonade}) and (\ref{apfelsaft}) we obtain
  \begin{align} \label{cola}
    \E\int_0^t|u(s,x)|_V^p\,ds &\leq \frac{1}{\alpha_1}
    \left(\alpha_0\E\int_0^t|u(s,x)|_H^2 ds + \E|x|_H^2
    + tC_0\right)\nonumber\\
    &\leq \frac{1}{\alpha_1} \left[\left(\frac{\alpha_0}{\gamma}+1\right)
          \E|x|_H^2 + t(\alpha_0K+C_0)\right]
  \end{align}
  for all $t>0$.

  We shall now use the estimates just obtained to prove the tightness
  of the sequence of measures
  \[
  \nu_n := \frac{1}{n}\int_0^n \lambda_t\,dt, \qquad n \in \enne,
  \]
  where $\lambda_t$ stands for the law of the random variable
  $u(t,0)$, so that
  \[
  \int_H \varphi\,d\nu_n = \frac1n \int_0^n \E\varphi(u(t,0))\,dt
  \]
  for all $\varphi \in B_b(H)$. By (\ref{cola}) we obtain
  \[
  \E\int_0^t|u(s,0)|_V^p\,ds \lesssim 1 + t \qquad \forall t>0,
  \]
  which in turn implies
  \begin{equation}\label{sahne}
    \int_H |y|_V^p\,\nu_n(dy) = \frac1n \int_0^n \E|u(s,0)|_V^p\,ds \lesssim 1
    \qquad\forall n\in\enne.
  \end{equation}
  By Markov's inequality we thus obtain
  \[
  \sup_{n\in\enne}\nu_n(|y|_V\geq R) \leq \sup_{n\in\enne}\frac{1}{nR^p}
  \int_0^n \E|u(s,0)|_V^p\,ds \lesssim \frac{1}{R^p},
  \]
  which converges to zero as $R \to \infty$.  Since the ball
  $B_R:=\{y \in H: |y|_V \leq R\}$ is bounded in $V$, and $V
  \hookrightarrow H$ is compact, it follows that, for any given
  $\varepsilon$, there exists $\bar{R}\in\erre_+$ such that
  $\nu_n(B_{\bar{R}}) > 1 - \varepsilon$ uniformly over $n$, with
  $B_{\bar{R}}$ a compact subset of $H$. In other words, the sequence
  $\nu_n$ is tight, and Prohorov's theorem yields the existence of a
  subsequence $\nu_{n_k}$ such that $\nu_{n_k}\wto \nu$ in
  $\mathcal{M}_1(H)$.  Furthermore, recalling that $P_t$ is Feller on
  $H$, $\nu$ is an invariant measure for $P_t$ by Krylov-Bogoliubov's
  theorem.

  Let us now show that $\nu$ is concentrated on $V$. To this end,
  let $\Theta:H\rightarrow[0,\infty]$ be a lower semicontinuous
  function such that
  \[
  \Theta(y) =
  \begin{cases}
    |y|_V, &y \in V,\\
    +\infty, &y \in H \setminus V
  \end{cases}
  \]
  and $\Theta(y)=\sup_{k\in\enne} \big| \ip{\ell_k}{y} \big|$, where
  $\{\ell_k\}_{k\in\enne}$ is a countable dense subset of $B_1^{V'}
  \cap H$ in the topology of $H$ (see e.g. \cite[p.~74]{PreRoeck}),
  and $B_1^{V^\prime}$ is the closed unit ball in $V^\prime$. Then
  $(\ref{sahne})$ implies
  \begin{align*}
    \int_H \Theta(y)^p\,\nu(dy) &= \lim_{L\to\infty}\lim_{M\to\infty}
    \int_H \big(\sup_{k\leq L} |\ip{\ell_k}{y}|^p\wedge M \big)\,\nu(dy)\\
    &= \sup_{L,M\in\enne} \lim_{h\to\infty}
    \int_H \big(\sup_{k\leq L} |\ip{\ell_k}{y}|^p \wedge M
           \big)\,\nu_{n_h}(dy)\\
    &\leq \liminf_{h\to\infty} \sup_{L,M\in\enne}\int_H\big(
          \sup_{k\leq L} |\ip{\ell_k}{y}|^p \wedge M \big)\nu_{n_h}(dy)\\
    &= \liminf_{h\to\infty}\int_H |y|^p_V\,\nu_{n_h}(dy) < \infty,
  \end{align*}
  hence $\Theta<\infty$ $\nu$-a.e., thus also $\nu(V)=1$ since $\{
  y \in H:\,\Theta(y)<\infty\}=V$.
\end{proof}

\begin{thm}     \label{thm:intega}
  Let $\nu$ be an invariant measure for $P_t$. Then $\nu$ satisfies
  the estimate
  \[
    \int_H \big( |x|_H^2 + |x|_V^p \big)\,\nu(dx) < \infty.
  \]
\end{thm}
\begin{proof}
  Let $x \in H$ and consider the one dimensional process
  $U(t):=|u(t,x)|_H^2$, which can be written, in view of \ref{eq:ito}), as
  \[
  U(t) = |x|_H^2 + \int_0^t F_1(s)\,ds + \int_0^t\!\int_Z
  F_2(s,z)\,\bar{\mu}(ds,dz) + \int_0^t\!\int_Z F_3(s,z)\,\mu(ds,dz),
  \]
  where $F_1$, $F_2$, $F_3$ are defined in the obvious way. Let $\chi
  \in C^1_b(\erre_+,\erre)$ be a smooth cutoff function with
  $\chi(x)=1$ for all $x\in[0,1]$, $\chi(x)=0$ for all $x \geq 2$, and
  $\chi'(x)\leq 0$ for all $x\in\erre_+$. Setting
  $\chi_N(x)=\chi(x/N)$ and $\varphi_N(x)=\int_0^y \chi_N(y)\,dy$ for
  all $x \in \erre_+$, It\^o's formula yields, suppressing the
  $\cdot_H$ subscript for semplicity of notation,
  \begin{align}
  \varphi_N(U(t)) &= \varphi_N(|x|^2)
  + \int_0^t \varphi'_N(U(s-))\,dU(s)\nonumber\\
  &\quad + \sum_{s\leq t} \big[ \varphi_N(U(s-)+\Delta U(s))
  - \varphi_N(U(s-)) - \varphi'_N(U(s-))\Delta U(s) \big]
  \label{eq:rata}
  \end{align}
  By Taylor's formula, there exists $\theta \in (0,1)$ such that the
  summand in the last term on the right-hand side can be written as
  \[
  \varphi''_N\big(U(s-) + \theta\Delta U(s)\big) \, |\Delta U(s)|^2,
  \]
  which is negative $\P$-a.s. because
  $\varphi''_N(x)=\chi'_N(x)=N^{-1}\chi'(x/N) \leq 0$ for all $x \in
  \erre_+$. Moreover, the second term on the right-hand side of
  (\ref{eq:rata}) can be written as
  \begin{gather*}
  -2\int_0^t \chi_N(|u(s)|^2)\ip{Au(s)}{u(s)}\,ds
  + 2\int_0^t\!\int_Z \chi_N(|u(s-)|^2)\ip{u(s-)}{G(u(s-),z)}
    \,\bar{\mu}(ds,dz)\\
  + \int_0^t\!\int_Z \chi_N(|u(s-)|^2)|G(u(s-),z)|^2\,\mu(ds,dz).
  \end{gather*}
  Therefore, taking expectation on both sides of (\ref{eq:rata}),
  recalling that the compensator of $\mu$ is $\mathrm{Leb}\otimes m$,
  we are left with
  \begin{align*}
    \E\varphi_N(|u(t)|^2) &\leq \E\varphi_N(|x|^2)
    -2\E\int_0^t \chi_N(|u(s)|^2)\ip{Au(s)}{u(s)}\,ds\\
    &\quad + \E\int_0^t\!\int_Z \chi_N(|u(s)|^2)|G(u(s),z)|^2\,m(dz)\,ds.
  \end{align*}
  Recalling (\ref{eq:milch}), Tonelli's theorem yields
  \begin{align*}
    &\E\varphi_N(|u(t)|^2)
    + \alpha_1\int_0^t \E\chi_N(|u(s)|^2) |u(s)|_V^p\,ds\\
    &\qquad \leq \E\varphi_N(|x|^2)
    + \alpha_0\int_0^t \E\chi_N(|u(s)|^2) |u(s)|^2\,ds + tC.
  \end{align*}
  for all $t\geq0$. Integrating both sides with respect to $\nu$ on
  $H$, applying again Tonelli's theorem, the definition of invariant
  measure, and setting $t=1$, we obtain
  \begin{equation}     \label{eq:ced}
  \alpha_1 \int_H \chi_N(|x|^2) |x|_V^p\,\nu(dx) \leq
  \alpha_0 \int_H \chi_N(|x|^2) |x|^2\,\nu(dx) + C.
  \end{equation}
  By the inequality $\varepsilon^2|x|^2 \leq \varepsilon^p|x|^p+1$ and
  the continuity of $V \hookrightarrow H$, we have
  \[
  \int_H \chi_N(|x|^2) |x|^2\,\nu(dx) \leq
  \varepsilon^{p-2}c^p \int_H \chi_N(|x|^2) |x|_V^p\,\nu(dx)
  + \varepsilon^{-2},
  \]
  hence
  \[
  \int_H \chi_N(|x|^2) |x|^2\,\nu(dx) \leq \varepsilon^{-2}
  + \frac{\varepsilon^{p-2}c^p}{\alpha_1} \Big(
  \alpha_0 \int_H \chi_N(|x|^2) |x|^2\,\nu(dx) + C \Big).
  \]
  Choosing $\varepsilon$ sufficiently small we get
  \[
  \int_H \chi_N(|x|^2) |x|^2\,\nu(dx) \lesssim 1,
  \]
  thus also, by the monotone convergence theorem, $\int_H
  |x|^2\,\nu(dx)<\infty$. This immediately yields the result, in view
  of (\ref{eq:ced}).
\end{proof}

The estimates just established allow one to deduce the existence of an
ergodic invariant measure.
\begin{coroll}
  There exists an ergodic invariant measure for the semigroup $P_t$.
\end{coroll}
\begin{proof}
  The last estimate in the proof of the previous theorem and
  (\ref{eq:ced}) allow to conclude that there exists a constant $N$,
  independent of $\nu$, such that
  \[
  \int_H |x|_V^p\,\nu(dx) < N
  \]
  for any invariant measure $\nu$. Denoting by $\mathcal{N} \subset
  \mathcal{M}_1(H)$ the set of invariant measures of $P_t$, Markov's
  inequality yields
  \[
  \sup_{\nu\in\mathcal{N}} \nu(|x|_V>R) \leq
  \frac1{R^p}   \sup_{\nu\in\mathcal{N}} \int_H |x|_V^p\,\nu(dx)
  < \frac{N}{R^p} \xrightarrow{R\to+\infty} 0.
  \]
  Therefore, by the same argument used in the proof of Theorem
  \ref{thm:exim}, we conclude that $\mathcal{N}$ is tight, hence,
  thanks to Prohorov's theorem, (relatively) compact in
  $\mathcal{M}_1(H)$. Since $\mathcal{N}$ is non-empty and convex,
  Krein-Milman's theorem ensures that $\mathcal{N}$ has extreme
  points, which are ergodic invariant measures for $P_t$ by a
  well-known criterion (see e.g. \cite[thm.~19.25]{AliBor}).
\end{proof}

Finally, we give a sufficient condition for uniqueness of an invariant
measure under an extra superlinearity assumptions on the couple
$(A,G)$.
\begin{prop}
  Assume that there exist $\eta>0$ and $\delta>0$ such that
  \[
  2\ip{Av-Aw}{v-w} - |G(v,\cdot)-G(w,\cdot)|^2_m
  \geq \eta |v-w|_H^{2+\delta},
  \qquad\forall v,w\in V.
  \]
  Then $P_t$ has a unique strongly mixing invariant measure.
\end{prop}
\begin{proof}
  Let $x,y\in H$. Then It\^o's formula for the square of the norm in
  $H$ implies, after taking expectations,
  \begin{align*}
    \E|u(t,x)-u(t,y)|^2 &+ 2\E\int_0^t
                             \ip{Au(s,x)-Au(s,y)}{u(s,x)-u(s,y)}\,ds\\
    &\leq |x-y|^2 + \E\int_0^t\int_Z|G(u(s,x),z)-G(u(s,y),z)|^2\,m(dz)\,ds\\
    &\leq |x-y|^2 - \eta \int_0^t \E|u(s,x)-u(s,y)|^{2+\delta}\,ds\\
    &\leq |x-y|^2 - \eta \int_0^t \big(\E|u(s,x)-u(s,y)|^2\big)^{1+\delta/2}\,ds
  \end{align*}
  for all $t>0$, where we have used Jensen's inequality in the last
  step. Since the solution $\zeta:\erre_+\to\erre_+$ of the ordinary
  differential equation
  \[
  \zeta' = -\eta \zeta^{1+\delta/2}, \qquad \zeta(0)=|x-y|^2
  \]
  is such that $\lim_{t\to\infty} \zeta(t)=0$ for all $x$, $y \in H$,
  we conclude by a standard comparison argument that
  $\E|u(t,x)-u(t,y)|^2 \to 0$ as $t\to\infty$.

  Let $\nu$ be an invariant measure for $P_t$. Then for any $f \in
  C^{0,1}_b(H)$ we have
  \begin{align*}
    \Big| P_tf(x) - \int_H f\,d\nu \Big|
    &= \Big| \int_H P_tf(x)\,\nu(dy) - \int_H P_tf(y)\,\nu(dy) \Big|\\
    &\leq \int_H |P_tf(x)-P_tf(y)|\,\nu(dy)\\
    &\leq |f|_{\lip} \int_H \big(\E|u(t,x)-u(t,y)|^2\big)^{1/2}\,\nu(dy).
  \end{align*}
  Since $(\E|u(t,x)-u(t,y)|^2)^{1/2} \leq |x-y|$ and $\int_H
  |x-y|\,\nu(dy)<\infty$, we can pass to the limit under the integral
  sign as $t \to \infty$ by the dominated convergence theorem, thus
  concluding that $|P_tf(x)-\int_Hf\,d\nu| \to 0$ as $t\to \infty$,
  and in particular that $\nu$ is the unique invariant measure.
  Moreover, since $C^1_b(H)$ is dense in $L_2(H,\nu)$, one has that
  for any $f \in L_2(H,\nu)$,
  \[
  \lim_{t \to \infty} P_tf(x) = \int_H f\,d\nu,\qquad x\in H,
  \]
  i.e. $\nu$ is strongly mixing (in particular ergodic) as required.
\end{proof}


\section{Essential $m$-dissipativity of the Kolmogorov
operator}\label{sec:Kol}
Denoting by $u(\cdot,x)$ the solution to the stochastic equation
(\ref{eq:caro}), we have proved in the previous section that the
semigroup
\[
P_tf(x) := \E f(u(t,x)), \qquad f \in B_b(H)
\]
admits a (not necessarily unique) invariant measure $\nu$. As is
well-known, $P_t$ can be extended to a strongly continuous Markovian
semigroup of contractions on $L_p(H,\nu)$, $p \geq 1$. In the
following we shall denote the extension of $P_t$ to $L_p(H,\nu)$
again by $P_t$.

Let us define the operator $(L,D(L))$ in $L_1(H,\nu)$ by
\begin{align*}
Lf(x) &= -\ip{Ax}{Df(x)} + \mathcal{I}f(x), \qquad x \in V,\\
\mathcal{I}f(x) &= \int_Z \big[
f(x+G(z)) - f(x) - \ip{Df(x)}{G(z)} \big]\,m(dz),\\
D(L) &= \big\{ f \in C^2_b(H) \cap C^1_b(V): \; \ip{Ax}{Df(x)} \in L_1(H,\nu)
\big\}.
\end{align*}
Note that the nonlocal term $\mathcal{I}f$ in the definition of $L$ is
a well-defined element of $L_1(H,\nu)$ for $f \in C^{1,1}_b(H)$. In fact,
the fundamental theorem of calculus yields
\begin{align*}
\big| f(x+G(z))&-f(x)-\ip{Df(x)}{G(z)} \big|\\
&\leq \Big| \int_0^1 \ip{Df(x+\theta G(z))}{G(z)}\,d\theta
           - \ip{Df(x)}{G(z)} \Big|
\leq |Df|_{\lip} |G(z)|^2
\end{align*}
therefore, since $G \in L_2(Z,m)$, we have that $|\mathcal{I}f|
\lesssim 1$, thus also $\mathcal{I}f \in L_1(H,\nu)$.

By a computation based on It\^o's formula one can see that the
infinitesimal generator of $P_t$ in $L_1(H,\nu)$ acts on smooth enough
functions as the operator $L$ just defined.  Since $P_t$ is a
contraction for all $t \geq 0$, we have that $(L,D(L))$ is dissipative
in $L_1(H,\nu)$. The question of $L_1$-uniqueness then arises
naturally: is $P_t$ the only strongly continuous semigroup on
$L_1(H,\nu)$ such that its infinitesimal generator extends $(L,D(L))$?
Under a ``regularizability'' hypothesis on $A$, we shall give an
affirmative answer to this question, proving that the closure of $L$
in $L_1(H,\nu)$ generates a strongly continuous semigroup. In fact,
since $L$ is dissipative, this will imply that the semigroup coincides
with $P_t$.

Throughout this section we shall assume that there exists a sequence
of monotone operators $A^\varepsilon \in \dot{C}^{0,1}(H \to H) \cap
C^1_b(V \to V')$ such that $A^\varepsilon x \to Ax$ in $V'$ for all $x
\in V$ and $|A^\varepsilon x|_{V'} \leq N(|x|_V^{p-1}+1)$ with $N$
independent of $\varepsilon$.

We are going to prove that $L$ is dissipative in $L_1(H,\nu)$
just assuming that $\nu$ is an infinitesimally invariant for $L$
satisfying the integrability condition
\begin{equation}     \label{eq:ippo}
x \mapsto |x|_V^p + |x|_H \in L_1(H,\nu).
\end{equation}
More precisely, the assumption of $\nu$ being infinitesimally
invariant amounts to assuming that
\[
\int_H Lf\,d\nu = 0 \qquad \forall f \in \mathcal{K},
\]
where $\mathcal{K}:=C^{1,1}_b(H) \cap C^1_b(V')$.  Note that
\eqref{eq:ippo} and $f \in \mathcal{K}$ imply that $Lf \in
L_1(H,\nu)$, so that the above condition is meaningful. In fact, one
has $\mathcal{I}f \in L_1(H,\nu)$ for all $f \in C^{1,1}_b(H)$, as seen
above, and
\[
|\ip{Ax}{Df(x)}| \leq |Ax|_{V'} \sup_{y\in V}|Df(y)| \lesssim |x|_V^p + 1
\in L_1(H,\nu).
\]
Let us recall that any invariant measure is infinitesimally invariant,
but the converse does not hold, in general. Moreover, any invariant
measure for (\ref{eq:caro}) satisfies the integrability condition
(\ref{eq:ippo}) thanks to Theorem \ref{thm:intega}.
\begin{lemma}
  The operator $(L,D(L))$ is dissipative, hence closable, in $L_1(H,\nu)$.
\end{lemma}
\begin{proof}
  Let $f \in \mathcal{K}$ and $\gamma_\varepsilon \in C^2(\erre)$ be a
  convex function with such that $\gamma'_\varepsilon$ is a smooth
  approximation of the signum graph
  \[
  \mathrm{sgn}(x) =
  \begin{cases}
    -1, & x<0,\\
    [-1,1], & x=0,\\
    1, & x>0.
  \end{cases}
  \]
  Then we have $\gamma_\varepsilon(f) \in \mathcal{K}$ and
  \begin{equation}     \label{eq:chino}
  L\gamma_\varepsilon(f) = \ip{Ax}{Df}\gamma'_\varepsilon(f) 
  + \mathcal{I}\gamma_\varepsilon(f),
  \end{equation}
  where, by a direct calculation,
  \begin{align*}
  &\mathcal{I}\gamma_\varepsilon(f) - \gamma'_\varepsilon(f) \mathcal{I}f\\
  &\qquad = \int_Z \big[\gamma_\varepsilon(f(x+G(z))) - \gamma_\varepsilon(f(x))
    - \gamma'_\varepsilon(f(x))\big(f(x+G(z))-f(x)\big)\big]\,m(dz)\\
  &\qquad =: R_\varepsilon(f).
  \end{align*}
  Since $\gamma_\varepsilon$ is convex and differentiable, we infer
  that $R_\varepsilon(f) \geq 0$. Therefore, taking the previous
  inequality into account and the infinitesimal invariance of $\nu$,
  one has, integrating (\ref{eq:chino}) with respect to $\nu$,
  \[
  \int L\gamma_\varepsilon(f)\,d\nu = 0 = \int \gamma'_\varepsilon(f)\,Lf\,d\nu
  + \int R_\varepsilon(f)\,d\nu,
  \]
  hence $\int \gamma'_\varepsilon(f)\,Lf\,d\nu \leq 0$, and passing to
  the limit as $\varepsilon \to 0$,
  \[
  \int Lf\,\xi\,d\nu \leq 0,
  \]
  where $\xi \in L_\infty(H,\nu)$, $\xi \in \mathrm{sgn}(f)$ $\nu$-a.e.
  Since $L_1(H,\nu)'=L_\infty(H,\nu)$, recalling that the duality map
  $J:L_1(H,\nu) \to 2^{L_\infty(H,\nu)}$ is given by
  \[
  J: u \mapsto \big\{ v \in L_\infty(H,\nu): \;
  v \in |u|_{L_1(H,\nu)} \mathrm{sgn}(u) \quad \nu\text{-a.e.} \big\},
  \]
  we infer by the previous inequality that $L$ is dissipative in
  $L_1(H,\nu)$.
\end{proof}
The following result gives a positive answer to the $L_1$-uniqueness
question posed above.
\begin{thm}
  Let $(\bar{L},D(\bar{L}))$ be the closure of the Kolmogorov operator
  $L$ in $L_1(H,\nu)$. Then $(\bar{L},D(\bar{L}))$ generates a strongly
  continuous Markovian semigroup of contractions $T_t$ in
  $L_1(H,\nu)$, for which $\nu$ is an invariant measure.
\end{thm}
\begin{proof}
  By the Lumer-Phillips theorem, $\bar{L}$ generates a strongly
  continuous semigroup of contractions if $R(\alpha I-\bar{L})$ is
  dense in $L_1(H,\nu)$ for some $\alpha>0$.

  Consider the regularized equation
  \begin{equation}     \label{eq:reg}
  du(t) + A^{\varepsilon\lambda} u\,dt = \int_Z G(z)\,d\bar{\mu}(dt,dz),
  \qquad u(0)=x \in H,
  \end{equation}
  with
  \[
  A^{\varepsilon\lambda} x := \int_H e^{\lambda C} A^\varepsilon(e^{\lambda C}x+y)
  \gamma_{\frac{1}{2}C^{-1}(e^{2\lambda C}-1)}(dy), \qquad \lambda>0,
  \]
  where $C: D(C) \subset V \to H$ is a self-adjoint, negative definite
  linear operator such that $C^{-1}$ is of trace class, and $\gamma_Q$
  stands for a centered Gaussian measure on $H$ with covariance
  operator $Q$.  Then, by the Cameron-Martin formula, one has
  \[
  A^{\varepsilon\lambda} \in C^\infty(H \to H),
  \qquad
  (A^{\varepsilon\lambda})' \in C_b^\infty(H \to \mathcal{L}(H\to H))
  \]
  and $A^{\varepsilon\lambda}x \to A^\varepsilon x$ for all $x \in H$
  as $\lambda \to 0$ (see e.g. \cite[{\S}2.3-2.4]{DP-K} for
  details). Moreover, $A^{\varepsilon\lambda}$ inherits the
  monotonicity of $A^\varepsilon$, and
  \[
  (A^{\varepsilon\lambda})' x = \int_H e^{\lambda C}
  (A^\varepsilon)'(e^{\lambda C}x+y) e^{\lambda C}
  \gamma_{\frac{1}{2}C^{-1}(e^{2\lambda C}-1)}(dy),
  \]
  so that $A^{\varepsilon\lambda} \in C^1_b(V \to V')$.

  Since $A^{\varepsilon\lambda}$ is Lipschitz continuous on $H$, (\ref{eq:reg})
  admits a unique strong solution $u_{\varepsilon\lambda}$ (e.g. by
  \cite[thm.~34.7]{Met}). Set
  \begin{equation}     \label{eq:resolv}
  f_{\varepsilon\lambda}(x) := \E\int_0^\infty e^{-\alpha t}
  \varphi(u_{\varepsilon\lambda}(t,x))\,dt,\qquad x\in H,
  \end{equation}
  where $\varphi \in \mathcal{K}$ and $\alpha>0$ are fixed.
  Since $A^{\varepsilon\lambda} \in C^1(H \to H)$, one has, thanks to
  \cite[thm.~36.9]{Met}, that $x \mapsto u_{\varepsilon\lambda}(t,x)$ is
  Fr\'echet differentiable for all $t \geq 0$, and its Fr\'echet
  derivative acting on an arbitrary $y \in H$, denoted by
  $v^y_{\varepsilon\lambda}:=Du_{\varepsilon\lambda}[y]$, solves the initial value
  problem (in the $\P$-a.s. sense)
  \begin{equation}     \label{eq:deri}
  \frac{d}{dt}v^y_{\varepsilon\lambda}
     + (A^{\varepsilon\lambda})'(u_{\varepsilon\lambda})
  v^y_{\varepsilon\lambda}=0,
  \qquad v^y_{\varepsilon\lambda}(0,x)=y.
  \end{equation}
  A computation based on It\^o's lemma for the square of the norm and
  the monotonicity of $A^{\varepsilon\lambda}$ reveals that $x \mapsto
  u_{\varepsilon\lambda}(\cdot,x) \in \dot{C}^{0,1}(H \to \mathbb{H}_2(T))$ for
  all $T\geq 0$, and
  \[
  \big| x \mapsto u_{\varepsilon\lambda}(t,x) \big|_{\dot{C}^{0,1}(H \to H)}
  \leq 1 
  \qquad \forall t \geq 0.
  \]
  This immediately implies that $|v^y_{\varepsilon\lambda}| \leq |y|$ for all
  $y \in H$, as the operator norm of the Fr\'echet derivative of a
  Lipschitz continuous function cannot exceed its Lipschitz constant.
  Moreover, since $(A^{\varepsilon\lambda})'(\xi) \in C^0_b(H \to H)$ for all
  $\xi \in H$, from (\ref{eq:deri}) we infer that $x \mapsto
  u_{\varepsilon\lambda}(t,x)$ is continuously differentiable $\P$-a.s. for
  all $t\geq 0$ (e.g. by \cite[{\S}X.8]{dieudonne}). Applying the
  chain rule for Fr\'echet derivatives (see e.g.
  \cite[Prop.~1.4]{AmbPro}) in (\ref{eq:resolv}), taking into account
  that $\varphi \in C^{1,1}_b(H)$ and $u_{\varepsilon\lambda}$ is Fr\'echet
  differentiable with $|Du_{\varepsilon\lambda}(t)|$ bounded uniformly over
  $t$, we get
  \begin{equation}     \label{eq:diffres}
  Df_{\varepsilon\lambda}(x)[y] = \E\int_0^\infty e^{-\alpha t}
  D\varphi(u_{\varepsilon\lambda}(t,x)) v^y_{\varepsilon\lambda}(t,x)\,dt
  \end{equation}
  for all $y \in H$, which also immediately yields
  \begin{equation}     \label{eq:limuf}
  \big| Df_{\varepsilon\lambda}(x)[y] \big| \lesssim |y|
  \qquad \forall y \in H,
  \end{equation}
  that is $f_{\varepsilon\lambda} \in C^1_b(H)$. In order to conclude
  that $f_{\varepsilon\lambda} \in C^{1,1}_b(H)$ we thus have to prove
  that $Df_{\varepsilon\lambda} \in \dot{C}^{0,1}(H \to H)$. Let us
  observe that we can write
  \begin{align*}
  & \big| Df_{\varepsilon\lambda}(x)[y] - Df_{\varepsilon\lambda}(x)[z] \big|\\
  &\qquad\quad \leq \E\int_0^\infty e^{-\alpha t} \big|
  D\varphi(u_{\varepsilon\lambda}(t,x)) v_{\varepsilon\lambda}^y(t,x) 
  - D\varphi(u_{\varepsilon\lambda}(t,z))v_{\varepsilon\lambda}^y(t,z)
  \big|\,dt\\
  &\qquad\quad \leq \E\int_0^\infty e^{-\alpha t} \big|
  D\varphi(u_{\varepsilon\lambda}(t,x)) v_{\varepsilon\lambda}^y(t,x)
  - D\varphi(u_{\varepsilon\lambda}(t,x))v_{\varepsilon\lambda}^y(t,z)
  \big|\,dt\\
  &\qquad\quad\quad + \E\int_0^\infty e^{-\alpha t} \big|
  D\varphi(u_{\varepsilon\lambda}(t,x)) v_{\varepsilon\lambda}^y(t,z)
  - D\varphi(u_{\varepsilon\lambda}(t,z))v_{\varepsilon\lambda}^y(t,z) \big|\,dt,
  \end{align*}
  where, recalling that $x \mapsto u_{\varepsilon\lambda}(t,x)$ and
  $v_{\varepsilon\lambda}(t)$ are respectively Lipschitz and bounded
  uniformly over $\varepsilon$, $\lambda$ and $t$, and that $\varphi \in
  C^{1,1}_b(H)$,
  \begin{align*}
    \big| D\varphi(u_{\varepsilon\lambda}(t,x)) v_{\varepsilon\lambda}^y(t,z) -
    D\varphi(u_{\varepsilon\lambda}(t,z))v_{\varepsilon\lambda}^y(t,z) \big|
    &\leq
    |D\varphi|_{\lip} |u_{\varepsilon\lambda}(t,x)-u_{\varepsilon\lambda}(t,z)| \,
    |v_{\varepsilon\lambda}^y(t,z)|\\
    &\lesssim |x-z| \, |y|.
  \end{align*}
  Moreover, we also have
  \begin{align*}
    \big| D\varphi(u_{\varepsilon\lambda}(t,x)) v_{\varepsilon\lambda}^y(t,x) -
    D\varphi(u_{\varepsilon\lambda}(t,x))v_{\varepsilon\lambda}^y(t,z) \big| \leq
    |D\varphi|_{C^0(H\to H)} |v_{\varepsilon\lambda}^y(t,x) -
    v_{\varepsilon\lambda}^y(t,z)|,
  \end{align*}
  from which it follows that in order to show that $Df_{\varepsilon\lambda}$ is
  Lipschitz on $H$ it suffices to prove that $x \mapsto
  v_{\varepsilon\lambda}(t,x)$ is Lipschitz on $H$. We have
  \[
  \frac{d}{dt} \big(v_{\varepsilon\lambda}^y(t,x)
               - v_{\varepsilon\lambda}^y(t,z) \big)
  + (A^{\varepsilon\lambda})'(u_{\varepsilon\lambda}(t,x))
                            v_{\varepsilon\lambda}^y(t,x)
  - (A^{\varepsilon\lambda})'(u_{\varepsilon\lambda}(t,z))
    v_{\varepsilon\lambda}^y(t,z) = 0,
  \]
  hence, taking scalar products with $v_{\varepsilon\lambda}^y(t,x) -
  v_{\varepsilon\lambda}^y(t,z)$,
  \[
  \frac12 \frac{d}{dt} \big| v_{\varepsilon\lambda}^y(t,x) -
  v_{\varepsilon\lambda}^y(t,z) \big|^2 + I = 0,
  \]
  where $I \equiv I(\varepsilon,\lambda,t,x,z,y)$ satisfies
  \begin{align*}
    I &=
    \big\langle(A^{\varepsilon\lambda})'(u_{\varepsilon\lambda}(t,x))
    \big( v_{\varepsilon\lambda}^y(t,x) -
    v_{\varepsilon\lambda}^y(t,z)\big),
    v_{\varepsilon\lambda}^y(t,x) - v_{\varepsilon\lambda}^y(t,z) \big\rangle\\
    &\quad + \big\langle
    (A^{\varepsilon\lambda})'(u_{\varepsilon\lambda}(t,x))
    v_{\varepsilon\lambda}^y(t,z) -
    (A^{\varepsilon\lambda})'(u_{\varepsilon\lambda}(t,z))
    v_{\varepsilon\lambda}^y(t,z),
    v_{\varepsilon\lambda}^y(t,x) - v_{\varepsilon\lambda}^y(t,z) \big\rangle\\
    &\geq \big\langle
    (A^{\varepsilon\lambda})'(u_{\varepsilon\lambda}(t,x))
    v_{\varepsilon\lambda}^y(t,z)
    -
    (A^{\varepsilon\lambda})'(u_{\varepsilon\lambda}(t,z))
    v_{\varepsilon\lambda}^y(t,z),
    v_{\varepsilon\lambda}^y(t,x) - v_{\varepsilon\lambda}^y(t,z)
    \big\rangle,
  \end{align*}
  once one takes into account that
  $(A^{\varepsilon\lambda})'(u_{\varepsilon\lambda}(t,x))$ is a positive linear
  operator, because $A^{\varepsilon\lambda}:H \to H$ is monotone and
  differentiable. Then we also get, recalling that
  $|v_{\varepsilon\lambda}^y(t,z)| \leq |y|$,
  \begin{align*}
    -I &\leq \frac12 \big| \big(
    (A^{\varepsilon\lambda})'(u_{\varepsilon\lambda}(t,x)) -
    (A^{\varepsilon\lambda})'(u_{\varepsilon\lambda}(t,z)) \big)
    v_{\varepsilon\lambda}^y(t,z) \big|^2\\
    &\quad + \frac12 \big| v_{\varepsilon\lambda}^y(t,x) 
    - v_{\varepsilon\lambda}^y(t,z) \big|^2\\
    &\leq \frac12 |y|^2 \, [(A^{\varepsilon\lambda})']^2_1 \, \big|
    u_{\varepsilon\lambda}(t,x) - u_{\varepsilon\lambda}(t,z) \big|^2
    + \frac12 \big| v_{\varepsilon\lambda}^y(t,x)
    - v_{\varepsilon\lambda}^y(t,z) \big|^2\\
    &\lesssim |y|^2 \, |x-z|^2 + \big| v_{\varepsilon\lambda}^y(t,x) -
    v_{\varepsilon\lambda}^y(t,z) \big|^2.
  \end{align*}
  In the last step we have used that $(A^{\varepsilon\lambda})' \in
  C_b^{\infty}(H \to \mathcal{L}(H \to H)$ and that $x \mapsto
  u_{\varepsilon\lambda}(t,x)$ is Lipschitz. Gronwall's inequality
  then yields \[ \big| v_{\varepsilon\lambda}(t,x) -
  v_{\varepsilon\lambda}(t,z) \big| \lesssim |x-z|, \] thus concluding
  the proof that $f_{\varepsilon\lambda} \in C^{1,1}_b(H)$.

Let us now prove that $f_{\varepsilon\lambda} \in C^1_b(V')$: in view of
(\ref{eq:diffres}), it is enough to prove that
$|v_{\varepsilon\lambda}^y(x)|_{V'} \leq |y|_{V'}$. Here we regard $\varphi$ as
a function from $V'$ to $\erre$ and $x \mapsto u_{\varepsilon\lambda}(t,x)$ as
a map from $V'$ to itself, so that $v_{\varepsilon\lambda}(t,x) \in
\mathcal{L}(V' \to V')$ and $v_{\varepsilon\lambda}^y(t,x) \in V'$. Let $J:V' \to
V'' \simeq V$ denote the duality map between $V'$ and $V$ (or
equivalently, let $J=F^{-1}$, with $F$ the duality map between $V$ and
$V'$). Multiplying both sides of (\ref{eq:deri}) by
$J(v_{\varepsilon\lambda}^y(t,x))$, in the sense of the duality pairing between
$V'$ and $V$, we obtain, taking into account that $(A^{\varepsilon\lambda})'$
is positive, $|v_{\varepsilon\lambda}^y(x)|_{V'} \leq |y|_{V'}$. We have thus
proved that $f_{\varepsilon\lambda} \in \mathcal{K}$.  This in turn implies
that $f_{\varepsilon\lambda}$ satisfies
  \begin{multline*}
  \alpha f_{\varepsilon\lambda}(x) + {}_{V'}\big\langle A^{\varepsilon\lambda}x,
  Df_{\varepsilon\lambda}(x)\big\rangle_V\\
  - \int_Z \big[f_{\varepsilon\lambda}(x+G(z)) - f_{\varepsilon\lambda}(x) 
                - \ip{Df_{\varepsilon\lambda}(x)}{G(z)}\big]\,m(dz)
  = \varphi(x), \qquad x \in H,
  \end{multline*}
  hence also
  \[
  \alpha f_{\varepsilon\lambda}(x) + \ip{Ax}{Df_{\varepsilon\lambda}(x)}
  - \mathcal{I}f_{\varepsilon\lambda}(x) = \varphi(x)
  + \ip{Ax-A^{\varepsilon\lambda}x}{Df_{\varepsilon\lambda}(x)},
  \]
  and
  \[
  \big| \alpha f_{\varepsilon\lambda} + \ip{Ax}{Df_{\varepsilon\lambda}}
  - \mathcal{I}f_{\varepsilon\lambda} \big|_{L_1(H,\nu)} \leq 
  |\varphi|_{L_1(H,\nu)} 
  + \big| \ip{Ax-A^{\varepsilon\lambda}x}{Df_{\varepsilon\lambda}}
    \big|_{L_1(H,\nu)}.
  \]
  Note that $|Df_{\varepsilon\lambda}(x)|_V \lesssim 1$ thanks to the above
  bound on $|v_{\varepsilon\lambda}(x)|_{V'}$, so that
  \begin{align*}
  &\int_H \big|\ip{Ax-A^{\varepsilon\lambda}x}{Df_{\varepsilon\lambda}(x)}
         \big|\,\nu(dx)\\
  &\qquad\qquad \lesssim
  \int_H |Ax-A^{\varepsilon}x|_{V'}\,\nu(dx)
  + \int_H |A^{\varepsilon}x-A^{\varepsilon\lambda}x|_{V'}\,\nu(dx)
  \end{align*}
  which converges to $0$ as $\lambda \to 0$ and $\varepsilon \to 0$ by
  the dominated convergence theorem. In fact, thanks to the hypotheses
  on $A$ and $A^\varepsilon$, we have $|Ax-A^{\varepsilon}x|_{V'}
  \lesssim |x|_V^p+1$ for all $x \in V$, and $\nu$ is concentrated on
  $V$ by (\ref{eq:ippo}). Moreover, since $H \hookrightarrow V'$ is
  continuous and $|A^{\varepsilon\lambda}x| \leq |A^\varepsilon x|$
  for all $x \in H$, we have
  $|A^{\varepsilon}x-A^{\varepsilon\lambda}x|_{V'}| \lesssim |x|_H + 1
  \in L_1(H,\nu)$, because of (\ref{eq:ippo}).  We have thus shown
  that
  \[
  \lim_{\varepsilon\to 0}\,\lim_{\lambda\to 0}
  \big( \alpha f_{\varepsilon\lambda} + \ip{Ax}{Df_{\varepsilon\lambda}}
    - \mathcal{I}f_{\varepsilon\lambda} \big) = \varphi
  \]
  in $L_1(H,\nu)$, i.e. that $R(\alpha I - L)$ is dense in
  $L_1(H,\nu)$, because $\mathcal{K}$ is dense in $L_1(H,\nu)$.  Since
  $L$ is also dissipative, we immediately infer that $\bar{L}$ is
  $m$-dissipative in $L_1(H,\nu)$.

  Let us denote the strongly continuous semigroup of contractions on
  $L_1(H,\nu)$ with generator $\bar{L}$ by $T_t$.  Let us now prove
  that $T_t$ is Markovian: for this it is enough to show that
  \[
  \int_H \bar{L}f \, 1_{\{f>1\}}\,d\nu \leq 0
  \qquad \forall f \in D(\bar{L})
  \]
  (see e.g. \cite[p.~109]{Stannat-L1}). Let $\gamma_\varepsilon \in
  C^2(\erre)$ be a convex function such that $\gamma'_\varepsilon$ is
  a smooth approximation of $x \mapsto 1_{]1,+\infty[}(x)$. Then,
  proceeding as in the proof of the previous lemma, we obtain the
  claim for all $f \in \mathcal{K}$ first, and for all $f \in
  D(\bar{L})$ by density.

  In order to prove that $\nu$ is an invariant measure for $T_t$, let
  us observe that one has, by definition of infinitesimal invariance
  and by a density argument,
  \[
  \int_H \bar{L}f\,d\nu = 0 \qquad \forall f \in D(\bar{L}).
  \]
  Since $T_tf \in D(\bar{L})$ for all $t \geq 0$ if $f \in
  D(\bar{L})$, we have, by the infinitesimal invariance of $\nu$,
  \[
  \int_H T_tf\,d\nu = \int_H f\,d\nu + \int_0^t\int_H \bar{L}T_sf\,d\nu\,ds
  = \int_H f\,d\nu
  \]
  for all $f \in D(\bar{L})$, thus also for all $f \in L_1(H,\nu)$ by
  density.
\end{proof}
\begin{rmk}
  The theorem implies that if $\nu$ is an invariant measure to the
  stochastic equation (\ref{eq:caro}) satisfying the integrability
  condition, then for all $f \in B_b(H)$, one has that $T_tf$ is a
  $\nu$-version of $P_tf$ for all $t \geq 0$.
\end{rmk}
\begin{rmk}
  The dissipativity of $L$ in $L_2(H,\nu)$ is easier to prove: in
  fact, for $f \in \mathcal{K}$, we have
  \[
  L(f^2) = 2fLf + \Gamma(f,f),
  \]
  where
  \[
  \Gamma(f,f) = \int_Z |f(x+G(z))-f(x)|^2\,m(dz) \geq 0
  \]
  is the so-called carr\'e du champ operator associated to
  $\mathcal{I}$, which is defined as
  \[
  \Gamma(f,g) := \mathcal{I}(fg) - f\mathcal{I}g - g\mathcal{I}f
  \]
  and takes the form
  \[
  \Gamma(f,g) = \int_Z \big(f(x+G(z))-f(x)\big) \big(g(x+G(z))-g(x)\big)
  \,m(dz).
  \]
  In particular one has the integration by parts formula
  \[
  \int f\,Lf\,d\nu = -\int \Gamma(f,f)\,d\nu.
  \]
  However, as one might expect, one needs stronger integrability
  assumptions on $\nu$ to prove the essential $m$-dissipativity of
  $L$, e.g. (roughly) of the type $x \mapsto |Ax|^2 \in
  L_1(H,\nu)$. Such an assumption would in turn require the data of
  the problem to be much more regular.
\end{rmk}


\section{Applications}     \label{sec:ex}

\subsection{SDEs with monotone drift}
If $V=H=\erre^d$, so that (\ref{eq:caro}) reduces to an ordinary
stochastic differential equation with monotone drift, our results on
ergodicity can be recovered applying \cite[Thm.~2]{KryGyo1}, which
provides existence and uniqueness of strong solutions (even in a more
general situation than that treated here), and
\cite[Thm.~I.25]{Sko-asympt}, which establishes boundedness in
probability for the solution by a Lyapunov-type criterion. In our case
one can choose as Lyapunov function simply $V(x)=|x|^2$.

\subsection{Stochastic equations with drift in divergence form}
Let $D \subset \erre^d$ be a bounded domain with smooth boundary, and
set $H:=L_2(D)$, $V=\mathring{W}_p^1(D)$, $V'=W_q^{-1}(D)$, with
$p>2$, $p^{-1}+q^{-1}=1$. Note that $V \hookrightarrow H$ is compact
by a Sobolev embedding theorem (see e.g. \cite[Thm.~0.4]{AmbPro}).
Consider the operator $A: V \to V'$ defined by
\[
Au := -\div\big(a(\nabla u)\big),
\]
which must be interpreted, as usual, as
\[
\ip{Au}{v} = \int_D \ip{a(\nabla u)}{\nabla v}_{\erre^d}\,dx \qquad
\forall v \in V.
\]
Here $a \in C^0(\erre^d \to \erre^d)$ is a monotone function satisfying
the polynomial growth condition $|a(x)| \lesssim |x|^{p-1}+1$ and the
coercivity condition $xa(x) \gtrsim |x|^p-1$.

Let $a_\varepsilon \in \dot{C}^{0,1}(\erre^d \to \erre^d)$,
$\tilde{a}_\varepsilon(x)=\varepsilon^{-1}(x-(I+\varepsilon a)^{-1}x)$
be the Yosida approximation of $a$, set $a_\varepsilon =
\tilde{a}_\varepsilon \ast \zeta_\varepsilon$, where
$\{\zeta_\varepsilon\}$ is a standard sequence of mollifier (in
particular $a_\varepsilon \in C^\infty$, $a'_\varepsilon \in
C_b^\infty$), and define the operator $A^\varepsilon$ on smooth
functions as
\[
A^\varepsilon u = -(I-\varepsilon \Delta)^{-1}
\div\big(a_\varepsilon(\nabla (I-\varepsilon \Delta)^{-1}u)\big),
\]
where $\Delta$ stands for the Dirichlet Laplacian on $D$.
We are going to show that $A^\varepsilon$ satisfies the assumptions of
the previous section. For this we shall need some elliptic regularity
results, which we recall here (see e.g. \cite[{\S}8.5]{krylov-LectSob}
for details).
\begin{lemma}     \label{lm:kry}
  Let $f \in L_p(D)$, $p \geq 2$. Then there exists $\varepsilon_1$
  such that, for all $\varepsilon < \varepsilon_1$, there exists a
  unique solution $u \in \mathring{W}_p^1$ to the equation
  \[
  u - \varepsilon \Delta u = f
  \]
  on $D$ with Dirichlet boundary conditions. Moreover $u$ satisfies
  the estimate
  \[
  | u |_{L_p(D)} + \varepsilon^{1/2} | u |_{W_p^1(D)} \leq N
  | f |_{L_p(D)},
  \]
where $N$ does not depend on $\varepsilon$.
\end{lemma}
Let us first show that $A^\varepsilon$ is well-defined both as an
operator from $H$ to itself, as well as from $V$ to $V'$.
Using the notation
\[
v^{(\varepsilon)} = (I-\varepsilon \Delta)^{-1}v,
\]
we may write
\begin{equation}     \label{eq:ape}
\ip{A^\varepsilon u}{v} = \int_D
\ip{a_\varepsilon(\nabla \ep{u})}{\nabla \ep{v}}_{\erre^d}\,dx.
\end{equation}
Note that if $v \in H$, then $\ep{v} \in \mathring{W}_2^1$ and
\[
|\nabla\ep{v}|_H \leq |\ep{v}|_{W_2^1(D)} \lesssim_\varepsilon |v|_H.
\]
Moreover, since $a_\varepsilon$ is Lipschitz continuous, we have
\[
|a_\varepsilon(\nabla \ep{u})| \leq |a_\varepsilon(\nabla \ep{u})
- a_\varepsilon(0)| + |a_\varepsilon(0)|
\lesssim |\nabla \ep{u}| + |a_\varepsilon(0)|,
\]
thus also
\[
\big| \ip{A^\varepsilon u}{v} \big| \leq
|a_\varepsilon(\nabla \ep{u})|_H \, |\nabla \ep{v}|_H
\lesssim (|u|_H+a_\varepsilon(0))|v|_H,
\]
which shows that $A^\varepsilon$ is well-defined from $H$ to
itself. Similarly, if $u$, $v \in V=\mathring{W}_p^1(D)$, we have, by
H\"older's inequality,
\[
\big| \ip{A^\varepsilon u}{v} \big| \leq
|a_\varepsilon(\nabla \ep{u})|_{L_q(D)} \, |\nabla \ep{v}|_{L_p(D)}
\lesssim (|u|_V+1)|v|_V,
\]
where we have used again Lemma \ref{lm:kry} and $\|\cdot\|_{L_q(D)}
\lesssim \|\cdot\|_{L_p(D)}$ for $p>q$ and $D$ bounded. We have thus
shown that $A^\varepsilon$ is well-defined from $V$ to $V'$.

The monotonicity of $A^\varepsilon$, both as an operator from $H$ to
itself and from $V$ to $V'$ is immediate by (\ref{eq:ape}) and the
monotonicity of $a_\varepsilon$.

Let us now show that $A^\varepsilon$ is Lipschitz continuous on
$H$. In fact, taking into account Lemma \ref{lm:kry}, we have
\begin{align*}
\big| \ip{A^\varepsilon u - A^\varepsilon v}{w} \big|
&= \big| \big\langle a_\varepsilon (\nabla\ep{u}) - 
a_\varepsilon (\nabla\ep{v}), \nabla\ep{w} \big\rangle \big|\\
&\lesssim_\varepsilon | \nabla(\ep{v} - \ep{w})| \, |\nabla\ep{w}|
\lesssim_\varepsilon |u - v|_H \, |w|_H.
\end{align*}
Since $a_\varepsilon \in C^1$, a direct computation yields that
$A^\varepsilon$ is G\^ateaux differentiable from $V$ to $V'$ with
G\^ateaux differential
\begin{equation}     \label{eq:aped}
\big\langle (A^\varepsilon)'(u)[v],w \big\rangle =
\int_D \big\langle a'_\varepsilon(\nabla\ep{u}) \nabla\ep{v},\nabla\ep{w}
\big\rangle_{\erre^d}\,dx
\end{equation}
for all $u$, $v$, $w \in V$. Note that the integral is well defined
because $|a'_\varepsilon(x)| \lesssim 1$ for all $x \in \erre^d$,
since $a_\varepsilon$ is Lipschitz continuous. By a well-known
criterion, we can conclude that $A^\varepsilon \in C^1(V \to V')$ if
we show that $(A^\varepsilon)'$ in (\ref{eq:aped}) is continuous as a
map $V \to \mathcal{L}(V \to V')$. Let $u_n \to u$ in
$\mathring{W}_p^1(D)$ as $n \to \infty$: applying H\"older's
inequality and Lemma \ref{lm:kry} repeatedly, we obtain
\begin{align*}
  &\sup_{|v|_V \leq 1} \, \sup_{|w|_{V} \leq 1} 
  \big\langle (A^\varepsilon)'(u_n)[v]-(A^\varepsilon)'(u)[v],w
  \big\rangle\\
  &\qquad\qquad\leq |\nabla\ep{w}|_{L_p(D)} \,
  \big| \big(a'_\varepsilon(\nabla\ep{u}_n)
        - a'_\varepsilon(\nabla\ep{u})\big)\nabla\ep{v}
  \big|_{L_{p/(p-1)}(D)}\\
  &\qquad\qquad\lesssim |\nabla\ep{v}|_{L_p(D)}
  \big| a'_\varepsilon(\nabla\ep{u}_n)
        - a'_\varepsilon(\nabla\ep{u}) \big|_{L_{p/(p-2)}(D)}\\
  &\qquad\qquad\lesssim   \big| a'_\varepsilon(\nabla\ep{u}_n)
        - a'_\varepsilon(\nabla\ep{u}) \big|_{L_{p/(p-2)}(D)}
  \xrightarrow{n \to \infty} 0.
\end{align*}
In fact, since $a'_\varepsilon$ is Lipschitz, it follows that
$|a'_\varepsilon(x)| \lesssim |x|^{p-2}+1$, and $\nabla\ep{u}_n \to
\nabla\ep{u}$ in $L_p$ implies convergence a.e. on a subsequence, from
which we can conclude by the dominated convergence theorem (see
e.g. \cite[Thm.~1.2.6]{AmbPro} for complete details in a similar
situation). We have thus proved that $A^\varepsilon \in C^1(V
\to V')$.

We conclude proving that
\[
\lim_{\varepsilon \to 0} |A^\varepsilon u - Au|_{V'} = 0
\qquad \forall u \in V.
\]
We have
\begin{multline*}
\int_D \big| \big\langle a_\varepsilon(\nabla\ep{u}),\nabla\ep{w} \big\rangle
- \big\langle a(\nabla u),\nabla w \big\rangle \big|\,dx\\
\leq 
\int_D \big| \big\langle a_\varepsilon(\nabla\ep{u}) - a(\nabla u),
\nabla w \big\rangle \big|\,dx
+ \int_D \big| \big\langle a_\varepsilon(\nabla\ep{u}),
\nabla\ep{w} - \nabla w \big\rangle \big|\,dx.
\end{multline*}
Since $|a_\varepsilon(x)| \leq N(|x|^{p-1}+1)$ with $N$ independent of
$\varepsilon$, the second term on the right-hand side can be
majorized by
\[
\big| a_\varepsilon(\nabla\ep{u}) \big|_{L_q} \,
\big| \nabla\ep{w} - \nabla w \big|_{L_p}
\lesssim \big(\big| \nabla u \big|^{p-1}_{L_p} +1 \big)\,
\big| \nabla\ep{w} - \nabla w \big|_{L_p}
\xrightarrow{\varepsilon\to 0} 0,
\]
where we have used once again Lemma \ref{lm:kry}.
Since $\nabla\ep{u} \to \nabla u$ in $L_p$ as $\varepsilon \to 0$, we
can upgrade the convergence to a.e. convergence, passing to a
subsequence, still denoted by $\varepsilon$. By Egorov's theorem,
there exists $D_\delta \subset D$, $|D \setminus D_\delta| \leq
\delta$, such that $\nabla\ep{u} \to \nabla u$ uniformly on $D_\delta$
as $\varepsilon \to 0$.  Since $a_\varepsilon$ and its limit function
$a$ are continuous on $\erre^d$, we have
\[
\lim_{\varepsilon \to 0} \lim_{\eta \to 0}
a_\varepsilon(\nabla u^{(\eta)}) =
\lim_{\eta \to 0} \lim_{\varepsilon \to 0} a_\varepsilon(\nabla u^{(\eta)}) =
a(\nabla u)
\]
pointwise on $D_\delta$, hence by a diagonal extraction argument,
there exists a further subsequence of $\varepsilon$, still denoted by
$\varepsilon$, such that, by the dominated convergence theorem,
\[
\big| a_\varepsilon(\nabla\ep{u}) - a(\nabla u) \big|_{L_q(D_\delta)}
\xrightarrow{\varepsilon\to 0} 0.
\]
On the other hand, we have
\begin{align*}
\big| a_\varepsilon(\nabla\ep{u}) - a(\nabla u) \big|_{L_q(D \setminus D_\delta)}
&\lesssim \int_{D \setminus D_\delta} \big(|\nabla u|^p +1\big)\,dx\\
&\lesssim |D \setminus D_\delta| \, \big(|\nabla u|_{L_p}+1\big)
\leq \delta \big(|\nabla u|_{L_p}+1\big).
\end{align*}
Since $\delta$ is arbitrary, we conclude that the integral above
converges to zero as $\varepsilon \to 0$, thus finishing the proof.

\subsection{Stochastic porous media equations}
Let $D$, $\Delta$, $p$, $q$, and $\{\zeta_\varepsilon\}$ be defined as
in the previous subsection.  Set $V=L_p(D)$, $H=W_2^{-1}(D)$,
$V'=\Delta(L_q(D))$, so that $V \hookrightarrow H$ compactly by a
Sobolev embedding theorem (see e.g. \cite[Prop.~4.6]{Tri3}). The norm
in $\mathring{W}_2^{-1}(D)$ will be denoted by
$|\cdot|_{-1}$. Consider the operator
\begin{align*}
  A: V &\to V'\\
     u &\mapsto -\Delta\beta(u),
\end{align*}
where $\beta \in C^0(\erre)$ is increasing and satisfies
\[
x\beta(x) \gtrsim |x|^{p}-1, \qquad
|\beta(x)| \lesssim |x|^{p-1} + 1
\]
for all $x \in \erre$. Note that these conditions on $\beta$ imply
that $A$ is well-defined (see e.g. \cite[{\S}4.1]{PreRoeck} for
details). Set
\[
\beta_\varepsilon(x) = \tilde{\beta}_\varepsilon \ast \zeta_\varepsilon,
\qquad 
\tilde{\beta}_\varepsilon = 
-\varepsilon^{-1} \vee \beta(x) \wedge \varepsilon^{-1},
\]
so that $\beta_\varepsilon \in C^\infty_b$, and define the operator
\[
A^\varepsilon u := -\Delta(I-\varepsilon\Delta)^{-1}
\beta_\varepsilon\big((I-\varepsilon\Delta)^{-1}u\big)
\]
on smooth functions. Then $A^\varepsilon$ is well-defined as an operator
from $H$ to itself, since
\[
\ip{A^\varepsilon u}{w}_{-1} = \int_D \beta_\varepsilon(\ep{u}) \ep{w}\,dx
\leq \big| \ep{w} \big|_{L_2(D)} \,
\big| \beta_\varepsilon(\ep{u}) \big|_{L_2(D)}
\lesssim |w|_{-1} \big( |u|_{-1}+1 \big)
\]
for all $u$, $w \in \mathring{W}_2^{-1}(D)$, because
$\beta_\varepsilon$ is Lipschitz and $|\ep{u}|_{L_2(D)} \lesssim
|u|_{-1}$ (see e.g. \cite[Thm.~3.3.1]{barbu-pde}). A completely
analogous computation also shows that $A^\varepsilon \in
\dot{C}^{0,1}(H \to H)$.  Let us also show that $A^\varepsilon$ is
well-defined as an operator from $V$ to $V'$: for $u$, $w \in L_p(D)$,
H\"older's inequality yields
\begin{align*}
\ip{A^\varepsilon u}{w} = \int_D \beta_\varepsilon(\ep{u}) \ep{w}\,dx
\leq \big| \ep{w} \big|_{L_p(D)}
     \big| \beta_\varepsilon(\ep{u}) \big|_{L_q(D)}
\lesssim |w|_{L_p(D)} \big( |u|_{L_p(D)} + 1 \big),
\end{align*}
where we have used Lemma \ref{lm:kry} and the estimate
$|\beta_\varepsilon(x)| \leq |\beta(x)| \lesssim |x|^{p-1} + 1$. The
latter also immediately implies that $|A^\varepsilon x|_{V'} \leq
N(|x|_V^{p-1}+1)$, with $N$ independent of $\varepsilon$.

As in the previous subsection, it is not difficult to see that
$A^\varepsilon$ is G\^ateaux differentiable from $V$ to $V'$, with
G\^ateaux differential
\[
\big\langle (A^\varepsilon)'u[v],w \big\rangle =
\int_D \beta'_\varepsilon(\ep{u})\ep{v}\ep{w},
\qquad u,\,v,\,w \in L_p(D).
\]
The continuity of the G\^ateaux differential (hence the Fr\'echet
differentiability of $A^\varepsilon:V \to V'$) follows by an argument
similar to the one used in the previous subsection, and we shall be
more concise here: for $u_n \to u$ in $L_p(D)$, we have
\begin{align*}
\big\langle (A^\varepsilon)'(u_n)[v] - (A^\varepsilon)'(u)[v],w \big\rangle
&\leq \big| \ep{w} \big|_{L_p(D)} \big| [\beta'_\varepsilon(\ep{u}_n)
- \beta'_\varepsilon(\ep{u})] \ep{v} \big|_{L_{p/(p-1)}(D)}\\
&\lesssim |v|_{L_p(D)} \, |w|_{L_p(D)} \,
\big| \beta'_\varepsilon(\ep{u}_n)-\beta'_\varepsilon(\ep{u})
\big|_{L_{p/(p-2)}(D)}.
\end{align*}
We proceed now as above: since $\ep{u}_n \to \ep{u}$ a.e. along a
subsequence, we can appeal to the dominated convergence theorem, in
view of the obvious bound $|\beta'_\varepsilon(x)| \lesssim |x|^{p-2}+1$.

The proof that $A^\varepsilon u \to Au$ in $V'$ for all $u \in V$ as
$\varepsilon \to 0$ is completely similar to the corresponding proof
in the previous subsection, hence omitted.

\bibliographystyle{amsplain}
\bibliography{ref}

\end{document}